\documentclass[a4paper,12pt]{amsart}
\usepackage{amssymb,amsthm,amsmath} 
\usepackage[utf8]{inputenc}  
\usepackage{enumerate}     	
\usepackage{graphicx}
\usepackage{bbm}

\theoremstyle{definition}
\newtheorem{proposition}{Proposition}[]
\newtheorem{lemma}{Lemma}[]
\newtheorem{assumption}{Assumption}[]
\newtheorem{remark}{Remark}[]

\newcommand{\xz}{X^{\zeta}}
\newcommand{\R}{R_{\lambda}}

\title[Ergodic control of diffusions]{Ergodic control of diffusions with random intervention times}
\date{\today}
\vspace{5pt}
\author{Jukka Lempa}
\address{Jukka Lempa, Department of Mathematics and Statistics, University of Turku, FI - 20014 Turun Yliopisto, Finland, \texttt{jumile@utu.fi}}
\author{Harto Saarinen}
\address{Harto Saarinen, Department of Mathematics and Statistics, University of Turku, FI - 20014 Turun Yliopisto, Finland, \texttt{hoasaa@utu.fi}}

\begin{document}

\begingroup
\let\newpage\relax%
\maketitle
\endgroup

\begin{abstract}
We study an ergodic singular control problem with constraint of a regular one-dimensional linear diffusion. The constraint allows the agent to control the diffusion only at jump times of independent Poisson process. Under relatively weak assumptions, we characterize the optimal solution as impulse type control policy, where it is optimal to exert the exact amount of control needed to push the process to a unique threshold. Moreover, we discuss the connection of the present problem to ergodic singular control problems, and finally, illustrate the results with different well-known cost and diffusion structures.
\end{abstract}

\noindent \emph{Keywords:} Bounded variation control, Ergodic control, Diffusion process, Resolvent operator, Poisson process, Singular stochastic control

\vspace{10pt}

\noindent \textup{2010} \textit{Mathematics Subject Classification}: \textup{49N25, 49L20, 60J60}

\tableofcontents

\section{Introduction}

In many biological and economical control problems, the decision maker is faced with the situation where the information of the evolving system is not available all the time. Instead, the decision maker might observe the state of the system only at discrete times, for example daily or weekly. Thus, in the following, we model these times when the controller receives the information of the evolving system as jump times of a Poisson process with a parameter $\lambda$. It is assumed that the decision maker can only exert the control at these exogenously given times, in other words, he can not act in the dark. Also, we restrict ourselves to controls of impulse type. Whenever the control is used, the decision maker has to pay a cost which is directly proportional to the size of the impulse.  Otherwise, when there are no interventions, we assume that the system evolves according to one-dimensional linear diffusion $X$ that is independent of the Poisson process. In literature these types of restriction processes on the controllability of $X$ are often referred as \emph{constraints} or \emph{signals}, see \cite{Lempa2012, Lempa2014, MenaldiRobin2016, MenaldiRobin2017, MenaldiRobin2018}.

In the classical case, the decision maker has continuous and complete information, and hence controlling is allowed whenever the decision maker wishes. The objective criterion to be minimized is often either a discounted cost or an ergodic cost (average cost per unit time). Both the discounted cost problems and ergodic problems have been studied in the literature, but the ergodic problems have gotten a bit less attention. This is because they are often mathematically more involved. However, from point of view of many applications, this is a bit surprising, as the discounting factor is often very hard or impossible to be estimated. Also, outside of financial applications the discounting factor might not have a very clear interpretation. 

The simplest case in the classical setting is the one where controlling is assumed to be costless. As a result, the optimal policy is often a local time of $X$ at the corresponding boundaries, see \cite{Alvarez1999, Matomaki2012} for discounted problems and \cite{AlvarezHening2019, Alvarez2019, JackZervos2006} for ergodic problems. One drawback of this model is that the optimal strategies are often singular with respect to Lebesgue measure, which makes them unappealing for applications. One way to make the model more realistic is to add a fixed transaction cost on the control. Then the optimal policy is often a sequential impulse control where the decision maker chooses a sequence of stopping times $\{ \tau_1, \tau_2, \ldots \}$ to exert the control and corresponding impulse sizes $\{ \zeta_1, \zeta_2, \ldots \}$, see \cite{AlvarezLempa2008, Alvarez2004, HarrisonSellkeTaylor1983}. In addition, it is possible that the flow of information is continuous but imperfect. These type of problems, often referred as \emph{filtering problems}, are also widely studied, see \cite{Fleming1968, Picard1986} and \cite{BainCrisan2009} for a textbook treatment and further references. In this case, the disturbance in the information flow is assumed to be such that the decision maker sees the underlying process all the time, but only observes a noisy version of it. 

As in the model at hand, another possibility is to allow the decision maker to control only at certain discrete exogenously given times. These times can be for example multiples of integers as in \cite{Rogers2001, KushnerDupuis1992} or given by a signal process. Often, as in our model, the times between the arrivals of the signal process are assumed to be exponentially distributed, see \cite{RogersZane2002, Wang2001, Lempa2014}. In \cite{RogersZane2002}, this framework was used as a simple model for liquidity effects in a classical investment optimization problem. Paper \cite{Wang2001} investigates both discounted cost and ergodic cost criterion while tracking a Brownian motion under quadratic cost and \cite{Lempa2014} generalizes the discounted problem to a more general payoff and underlying structure. Related studies in optimal stopping are \cite{DupuisWang2002, Lempa2012}. In \cite{DupuisWang2002} the authors consider a perpetual American call with underlying geometric Brownian motion and in \cite{Lempa2012} the results are generalized to larger class of underlying processes. Studies related to more general signal processes are found in \cite{MenaldiRobin2016, MenaldiRobin2017, MenaldiRobin2018}. In these the signal process can be a general, not necessarily independent, renewal process and the underlying process is a general Markov-Feller process. There are also multiple, less related studies, where an underlying Poisson process brings a different friction to the model, by either affecting the structure of the underlying diffusion \cite{Guo2004, Jiang2008} or the payoff structure \cite{Alvarez2001, Lempa2012B, Lempa2017}.

The main contribution of the paper is that we allow the underlying stochastic process $X$ to follow a general one-dimensional diffusion process and also allow a rather general cost function. This is a substantial generalization of \cite{Wang2001}, where the case of Brownian motion with quadratic cost is considered. We emphasize this in the illustrations in section 5 by explicitly solving multiple examples with different underlying dynamics and cost functions.  These generalizations have not, to our best knowledge, been studied earlier in the literature. Furthermore, we are able to connect the problem to a related problem in optimal ergodic singular control \cite{Alvarez2019}. 

The rest of the paper is organized as follows. In section 2, we define the control problem and proof auxiliary results. In section 3, we first investigate the necessary conditions of optimality by forming the associated free boundary problem, followed by the verification. We connect the problem to a similar problem of singular control in section 4 and then illustrate the results by explicitly solving few examples in section 5. Finally, section 6 concludes our study.

\section{The Control Problem}

\subsection{The underlying dynamics}

Let $(\Omega, \mathcal{F}, \{\mathcal{F}_t\}_{t \geq 0}, \mathbb{P})$ be a filtered probability space which satisfies the usual conditions. We consider an uncontrolled process $X$ defined on $(\Omega, \mathcal{F}, \{\mathcal{F}_t\}_{t \geq 0}, \mathbb{P})$, which evolves in $\mathbb{R}_+$, and is modelled as a solution to regular linear Itô diffusion
\begin{equation*}
dX_t=\mu(X_t)dt+ \sigma(X_t)dW_t, \quad \quad X_0 = x,
\end{equation*} 
where $W_t$ is the Wiener process and the functions $\mu,\sigma: (0,\infty) \to \mathbb{R}$ are continuous and satisfy the condition $\int_{x-\varepsilon}^{x+\varepsilon} \frac{1+|\mu(y)|}{\sigma^2(y)}dy < \infty$. These assumptions guarantee that the diffusion has a unique weak solution (see \cite{KaratzasShreve1991} section 5.5). Even though we consider the case where the process evolves in $\mathbb{R}_+$, we remark that it is done only for notional convenience, and the results would remain the same with obvious changes even if the state space would be replaced with any interval of $\mathbb{R}$.

We define the second-order linear differential operator  $\mathcal{A}$ which represents the infinitesimal generator of the diffusion $X$ as
\begin{equation*}
\mathcal{A} = \mu(x) \frac{d}{dx} + \frac{1}{2} \sigma^2(x)\frac{d^2}{dx^2},
\end{equation*}
and for a given $\lambda > 0$ we respectively denote the increasing and decreasing solutions to the differential equation $(\mathcal{A}-\lambda)f=0$ by $\psi_{\lambda} > 0$ and $\varphi_{\lambda} > 0$.

The differential operator $\lambda-\mathcal{A}$ has an inverse operator called the resolvent $R_{\lambda}$ defined by
\begin{equation*}
(R_{\lambda}f)(x)=\mathbb{E}_x \bigg[ \int_0^{\tau}e^{-\lambda s} f(X_s) ds \bigg]
\end{equation*}
for all $x \in \mathbb{R}_+$, and functions $f \in \mathcal{L}_1^{\lambda}$, where  $\mathcal{L}_1^{\lambda}$ is the set of functions $f$ on $\mathbb{R}_+$ which satisfy the integrability condition $\mathbb{E}_x [ \int_0^{\tau}e^{-\lambda s} |f(X_s)| ds ] < \infty $. Here $\tau$ is the first exit time from $\mathbb{R}_+$, i.e. $\tau = \inf \{ t \geq 0 \mid X_t \not \in \mathbb{R}_+  \}$.
Also define the scale density of the diffusion by
\begin{equation*}
S'(x) = \exp \bigg( - \int^x \frac{2 \mu(z)}{\sigma^2(z)}dz \bigg),
\end{equation*}
which is the (non-constant) solution to the differential equation $\mathcal{A} f=0$, and the speed measure of the diffusion by
\begin{equation*}
m'(x) = \frac{2}{\sigma^2(x)S'(x)}.
\end{equation*} 

It is well known, that the resolvent and the solutions $\psi_{\lambda}$ and $\varphi_{\lambda}$ are connected with the formula
\begin{align} \label{calculation formula for resolvent}
(R_{\lambda}f)(x) & = B^{-1}_{\lambda} \psi_{\lambda}(x)\int_x^{\infty}\varphi_{\lambda}(z)f(z)m'(z)dz  \\
& +  B^{-1}_{\lambda} \varphi_{\lambda}(x)\int_0^{x}\psi_{\lambda}(z)f(z)m'(z)dz, \nonumber
\end{align}
where 
\begin{equation*}
B_{\lambda} = \frac{\psi_{\lambda}'(x)}{S'(x)}\varphi_{\lambda}(x)-\frac{\varphi_{\lambda}'(x)}{S'(x)}\psi_{\lambda}(x)
\end{equation*}
denotes the Wronskian determinant (see \cite{BorodinSalminen} p.19). We remark that the value of $B_{\lambda}$ does not depend on the state variable $x$ because an application of harmonicity properties of $\psi_\lambda$ and $\varphi_\lambda$ yield
$$
\frac{dB_{\lambda}(x)}{dx} = 0.
$$ 
In calculations, it is sometimes also useful to use the identity 
\begin{equation} \label{m integraali skaalana}
\int_x^y \mu(z)m'(z)dz = \frac{1}{S'(y)}-\frac{1}{S'(x)}.
\end{equation}

\subsection{The control problem}

We consider a control problem where the goal is to minimize the average cost per unit time so that the controller is only allowed to control the underlying process at exogenously given times. These times, that the controller is allowed to use the control, are given as the arrival times of independent Poisson process, called the \emph{signal process} or \emph{constraint}, and thus the interarrival times are exponentially distributed.

\begin{assumption} The Poisson process $N_t$ and the controlled process $X_t$ are assumed to be independent, and the process $N_t$ is $\{\mathcal{F}_t\}_{t \geq 0}$-adapted. 
\end{assumption}

More precisely, the set of admissible controls $\mathcal{Z}$ is given by those non-decreasing left-continuous processes $\zeta_{t \geq 0}$ that have the representation
$$
\zeta_t=\int_{[0,t)} \eta_s dN_s,
$$
where $N$ is the signal process and the integrand $\eta$ is $\{\mathcal{F}_t\}_{t \geq 0}$-predictable.  The controlled dynamics are then given by the Itô integral
\begin{equation*}
\xz_t = X_0 + \int_0^{\tau_0^\zeta \wedge t} \mu(\xz_s) ds +  \int_0^{\tau_0^\zeta \wedge t} \sigma(\xz_s)dW_s - \zeta_t, \quad 0 \leq t \leq \tau_0^\zeta,
\end{equation*}
where $\tau_0^\zeta$ is the first exit time of $\xz_t$ from $\mathbb{R}_+$, i.e. $\tau_0^\zeta = \inf \{ t \geq 0 \mid \xz_t \not \in \mathbb{R}_+ \}$.

Define the average cost per unit time or ergodic cost criterion as
\begin{equation*}
J(x,\zeta):= \liminf_{T \to \infty} \frac{1}{T} \mathbb{E}_x \left[ \int_0^{T} (\pi(\xz_s)ds + \gamma d\zeta_s) \right],
\end{equation*}
where $\gamma$ is a given positive constant and $\pi: \mathbb{R}_+ \to \mathbb{R}$ is a function measuring the cost from continuing the process. Now define the value function
\begin{equation} \label{minimoitavan kriteerin arvo}
V(T,x) = \inf_{\zeta \in \mathcal{Z}} \mathbb{E}_x \left[ \int_0^T (\pi (X_s)ds + \gamma d \zeta_s) \right]
\end{equation}
and denote by $\beta$ the minimum average cost. The objective of the control problem is to minimize $J(x, \zeta)$ over all the admissible controls $\zeta \in \mathcal{Z}$ and to find, if possible, the optimal control $\zeta^*$ such that $\beta = \inf_{\zeta \in \mathcal{Z}}  J(x,\zeta) = J(x,\zeta^*)$.

We now define the auxiliary functions $\pi_{\gamma}:\mathbb{R}_+ \to \mathbb{R}$
\begin{equation*}
\pi_{\gamma}(x) = \pi(x)+ \gamma \lambda x
\end{equation*} 
and $\pi_{\mu}:\mathbb{R}_+ \to \mathbb{R}$,
\begin{equation*}
\pi_{\mu}(x) = \pi(x)+\gamma \mu(x).
\end{equation*}
In order for our solution to be well-behaved, we must pose some assumptions which are collected below.
\begin{assumption} \label{assumption 2}
\begin{enumerate}[($i$)] We assume that
\item the lower boundary $0$ and upper boundary $\infty$ are natural,
\item the cost $\pi$ is continuous, non-negative and minimized at $0$,
\item the function $\pi_{\mu}$ and $\text{id}:x \mapsto x$ are in $\mathcal{L}_1^{\lambda}$,
\item there exists a unique state $x^* \in \mathbb{R}_+$ such that $\pi_{\mu}$ is decreasing on $(0, x^*)$ and increasing on $[x^*, \infty)$. Also, $\lim_{x \uparrow \infty} \pi_{\mu}(x) > 0$.
\end{enumerate}
\end{assumption}

The boundaries of the state space are assumed to be natural, which means that, in the absence of interventions, the process can not become infinitely large or infinitely close to zero in finite time. In biological applications these boundary conditions guarantee that the population does not explode or become extinct in the absence of harvesting.  We refer to \cite{BorodinSalminen}, pp.18-20, for more thorough discussion of boundary behaviour of one-dimensional diffusions. Also, it is worth to mention, that no second order properties of $\pi_{\mu}$ are assumed.

In addition, the following limiting and integrability conditions on the scale density and speed measure must be satisfied. These conditions assure the existence of a stationary distribution of the underlying diffusion.
\begin{assumption} \label{assumption 3}
\begin{enumerate}[($i$)] We assume that
\item $m(0,y)=\int_{0}^{y} m'(z)dz <\infty$ and $\int_{0}^{y} \pi_{\mu}(z) m'(z)dz<\infty$ for all $y \in \mathbb{R}_+$.
\item $\lim_{x \downarrow 0} S'(x) = \infty$.
\end{enumerate}
\end{assumption}
\begin{remark}  \label{remark alareunasta}
The conditions of assumption \ref{assumption 3} alone guarantee that the lower boundary $0$ should be either natural or entrance, and hence unattainable. However, in the proof of lemma \ref{L ja H juuret} we must exclude the possibility of entrance to assure that $L(x)$ (defined below) attains also negative values. If we want to include this possibility, we would also have to assume that $\lim_{x \to 0} \pi_{\mu}(x) = \infty$, see the proof of lemma \ref{L ja H juuret}.
\end{remark}

\subsection{Auxiliary results}
Define the auxiliary functions $L: \mathbb{R_+} \to \mathbb{R}$ and $H: \mathbb{R_+} \to \mathbb{R}$ as
\begin{align*}
& L(x)= \lambda \int_x^\infty \pi_{\mu}(z)\varphi_{\lambda}(z)m'(z)dz + \frac{\varphi_\lambda'(x)}{ S'(x)}\pi_\mu(x), \\
& H(0,x) =  \int_{0}^x \pi_{\mu}(z)m'(z)dz - \pi_{\mu}(x)m(0,x).
\end{align*}
These functions will offer a convenient representation of the optimality equation in the section 3, and thus, their properties play a key role when determining the optimal control policy.
\begin{lemma} \label{L ja H juuret} Under the assumption 1, the functions $L(x)$ and $H(0,x)$ satisfy the following conditions: There exists a unique $\tilde{x} < x^*$ and a unique $\hat{x} > x^*$ such that
\begin{enumerate}[($i$)]
\item $L(x) \, \substack{< \\ = \\ >} \, 0$ when $x \, \substack{< \\ = \\ >} \, \tilde{x}$,
\item $H(0,x) \, \substack{< \\ = \\ >} \, 0$ when $x \, \substack{> \\ = \\ <} \, \hat{x}$.
\end{enumerate}
\end{lemma}
\begin{proof}
The proof of the claim on $L$ is similar to that of Lemma 3.3 in \cite{Lempa2014}. However, to show that the results are in accordance with the remark \ref{remark alareunasta},  we need to adjust the argument on finding a point $x_1 < x^*$ such that $L(x_1)<0$. Thus, assume for a while that $\lim_{x \to 0} \pi_{\mu}(x) = \infty$, and that $x^* > y > x$. Then
\begin{align*}
L(x)-L(y) = & \lambda \int_x^y \pi_\mu(z) \varphi_\lambda (z) m'(z) dz \\ & + \bigg[ \pi_{\mu}(x)\frac{\varphi_\lambda'(x)}{S'(x)}-\pi_{\mu}(y)\frac{\varphi_\lambda'(y)}{ S'(y)} \bigg] \\
\leq &  \frac{\varphi_\lambda'(y)}{S'(y)}(\pi_\mu(x)-\pi_\mu(y)),
\end{align*}
which shows that $\lim_{x \to 0} L(x) = - \infty$.

To prove the second part, assume first that $y>x>x^*$. Since the function $\pi_\mu$ is increasing on $(x^*,\infty)$, we see that
\begin{align*}
& H(0,y)-H(0,x) = \\
& \int_{x}^y \pi_\mu(z)m'(z)dz - \pi_\mu(y)m(0, y)+\pi_\mu(x)m(0, x) \\
& < \pi_\mu(y)(m(x,y)-m(0, y)) +\pi_\mu(x)m(0, x) \\
& = m(0, x)(\pi_\mu(x) - \pi_\mu(y)) \\
& <0,
\end{align*}
proving that $H$ is decreasing on $(x^*, \infty)$. It also follows from
$$ 
H(0,y)-H(0,x) < m(0,x)(\pi_\mu(x)-\pi_\mu(y)),
$$ 
that $\lim_{y \to \infty} H(0,y) < 0$. Next assume that $x^* > y > x$. Because $\pi_\mu$ is decreasing on $(0, x^*)$, we find similarly that
\begin{align*}
& H(0,y)-H(0,x) = \\
& \int_{x}^y \pi_\mu(z)m'(z)dz - \pi_\mu(y)m(0, y)+\pi_\mu(x)m(0, x) \\
& > \pi_\mu(y)(m(x,y)-m(0, y)) +\pi_\mu(x)m(0, x) \\
& = m(0, x)(\pi_\mu(x) - \pi_\mu(y)), \\
& > 0,
\end{align*}
implying that $H$ is increasing on $(0, x^*)$. Furthermore, $H$ is positive when $x < x^*$. Hence, by continuity, $H$ has a unique root, which we denote by $\hat{x}$.
\end{proof}

\begin{proposition} \label{uniqueness of the equation}
There exists a unique solution $\bar{x} \in (\tilde{x}, \hat{x})$ to the equation
\begin{equation*}
  S'(x) m(0,x) L(x) = -\varphi_{\lambda}'(x)H(0,x).
\end{equation*}
\end{proposition}
\begin{proof}
Define the function 
\begin{equation*} 
P(x)= S'(x) m(0,x) L(x) + \varphi_{\lambda}'(x)H(0,x).
\end{equation*}
Assuming that $x_1 > \hat{x} > x^*$, we get by lemma \ref{L ja H juuret} that
\begin{align*} 
P(x_1) = & S'(x_1) m(0,x_1) L(x_1) + \varphi_{\lambda}'(x_1)H(0,x_1) \geq 0.
\end{align*}
Similarly, when $x_2 < \tilde{x} < x^*$ we have that
\begin{align} \label{G positiivinen ennen y^*}
P(x_1) = & S'(x_2) m(0,x_2) L(x_2) + \varphi_{\lambda}'(x_2)H(0,x_2) \leq 0.
\end{align}
By continuity, the function $P(x)$ must have at least one root. We denote one of these roots by $z$. 

To prove that the root $z$ is unique, we first notice that the naturality of the upper boundary implies that (\cite{BorodinSalminen} p.19)
$$
\lim_{x \to \infty} \frac{\varphi_{\lambda}'(x)}{S'(x)}=0.
$$
Hence 
\begin{equation} \label{natural boundary phi integraali}
-\frac{1}{\lambda} \frac{\varphi_{\lambda}'(y^*)}{S'(y^*)} = \int_{y^*}^{\infty} \varphi_{\lambda}(z)m'(z)dz.
\end{equation}
Thus, we see that the equation $P(x)=0$ is equivalent to 
\begin{equation*}
     \frac{\int_{x}^{\infty} \pi_{\mu}(y) \varphi_{\lambda}(y)m'(y)dy}{\int_{x}^{\infty} \varphi_{\lambda}(y)m'(y)dy} = \frac{\int_0^x \pi_{\mu}(y)m'(y)dy}{\int_0^x m'(y)dy}.
\end{equation*}
Now, differentiating the left-hand side yields
\begin{equation*}
\frac{\varphi_\lambda (x) m'(x) L(x)}{I(x)^2},
\end{equation*}
where $I(x) = \int_{x}^{\infty} \varphi_{\lambda}(y)m'(y)dy$. Differentiating the right-hand side, and evaluating it at $z$, we get by using the equation $P(z)=0$ that
\begin{align*}
& \frac{\pi_{\mu}(z)m'(z)}{m(0,z)}-\frac{\int_0^z \pi_{\mu}(y)m'(y)dy}{m(0,z)}\frac{m'(z)}{m(0,z)}.  \\ 
& = \frac{-m'(z)L(z)}{I(z)m(0,z)} 
\end{align*}
Because $L(y)>0$ in the region $(\tilde{x}, \hat{x})$, and all the other terms are positive everywhere, we find by comparing the derivatives that
\begin{equation*}
    \frac{-m'(z)L(z)}{I(z)m(0,z)} < \frac{\varphi_\lambda (z) m'(z) L(z)}{I(z)^2}.
\end{equation*}
Therefore, by continuity, the intersection between the curves $$I(x)^{-1}\int_{x}^{\infty} \pi_{\mu}(y) \varphi_{\lambda}(y)m'(y)dy \, $$ and $$ \, m(0,x)^{-1}\int_0^x \pi_{\mu}(y)m'(y)dy$$ is unique. This unique point is denoted by $\bar{x}$.
\end{proof}

In the next lemma, we make some further computations that are needed for the sufficient conditions of the control problem. Define the functions $J:\mathbb{R}_+ \to \mathbb{R}$ and $I:\mathbb{R}_+ \to \mathbb{R}$
\begin{align*}
J(x) & = \frac{\gamma - (R_{\lambda} \pi_{\gamma})'(x)}{\varphi'_{\lambda}(x)}, \\
I(x) & = \frac{\int_{0}^x \pi_{\mu}(x) m'(t) dt}{m(0,x)}. 
\end{align*}

\begin{lemma} \label{lemma for variational}
Under the assumption 2:
\begin{enumerate}[($i$)]
    \item $J'(x) \substack{> \\ = \\ <} 0$ when $x \substack{> \\ = \\ <} \tilde{x}$,
    \item $I'(x) \substack{> \\ = \\ <} 0$ when $x \substack{> \\ = \\ <} \hat{x}$.
\end{enumerate}
Here $\tilde{x}$ and $\hat{x}$ are as in lemma \ref{L ja H juuret}.
\end{lemma}
\begin{proof}
The first claim follows from the formula
\begin{equation*}
J'(x) = \frac{2S'(x)}{\sigma^2(x)\varphi'_{\lambda}(x)^2} L(x)
\end{equation*}
which can be derived using representation (\ref{calculation formula for resolvent}) and straightforward differentiation (see \cite{Lempa2014} lemma 3 for details).
The claim on $I$ follows similarly as differentiation yields
\begin{equation*}
     I'(x) = -\frac{m'(x)}{m^2(0,x)} H(x).
\end{equation*}
\end{proof}

\section{The Solution}

\subsection{Necessary conditions}

Denote the candidate solution for \eqref{minimoitavan kriteerin arvo} as $F(T,x)$. We use the heuristic that $F(T,x)$ can be separated for large $T$ as
\begin{equation}\label{Heuristic}
F(T,x) \sim  \beta T + W(x).
\end{equation}
In mathematical finance literature, the constant $\beta$ usually denotes the minimum average cost per unit time and $W(x)$ is the potential cost function (see \cite{Fleming1999, SethiZhang2005}). The fact that the leading term $\beta T$ is independent of $x$ is, of course, dependent on the ergodic properties of the underlying process. We also note that this heuristic can be used as a separation of variables to solve a partial differential equation of parabolic type related to the expectation in (\ref{minimoitavan kriteerin arvo}) via the Feynman-Kac formula, see \cite{FlemingMcEneaney1995}.

We shall proceed as in \cite{Lempa2014}. We assume that the optimal control policy exists and is given by the following: When the process is below some threshold $y^*$ (called the \textit{waiting region}) we let the process run but if the process is above the threshold value $y^*$ (called the \textit{action region}) and the Poisson process jumps we exert the exact amount of control to push the process back to the boundary $y^*$ and start it anew.

In the waiting region $[0,y^*]$ we expect that the candidate solution satisfies the Bellman's principle  
\begin{equation}
F(T,x)=\mathbb{E}\left[ \int_0^U \pi (X_s)ds + F(T-U,X_U) \right],
\end{equation}
where $U$ is exponentially distributed random variable with mean $1/\lambda$. 

Using the heuristic \eqref{Heuristic} and noticing the connection between the random times $U$ and resolvent, we get by independence and strong Markov property that
\begin{align*}
& \mathbb{E}\left[ \int_0^U \pi (X_s)ds + F(T-U,X_U) \right] \\
 &= \lim_{r \to 0}(R_r \pi)(x) + \mathbb{E}_x [W(X_U)] - \frac{\beta}{\lambda} + \beta T -\mathbb{E}_x\left[ \int_U^{\infty} \pi (X_s)ds \right] \\
  &= \lim_{r \to 0}(R_r \pi)(x) + \lambda (R_{\lambda}W)(x) - \frac{\beta}{\lambda} + \beta T -\mathbb{E}_x\left[ \lim_{r \to 0}(R_r \pi)(X_U) \right] \\
   &= \lim_{r \to 0}(R_r \pi)(x) + \lambda (R_{\lambda}W)(x) - \frac{\beta}{\lambda} + \beta T - \lambda \lim_{r \to 0}(R_{\lambda} R_r \pi)(x).
\end{align*}
Hence, we arrive at the equation
\begin{equation*}
W(x) - \lim_{r \to 0}(R_r \pi)(x) =  \lambda R_{\lambda}(W(x) - \lim_{r \to 0}(R_r\pi)(x)) - \frac{\beta}{\lambda} .
\end{equation*}
We next choose $f(x) = W(x) - \lim_{r \to 0}(R_r \pi)(x)$ in lemma 2.1 of \cite{Lempa2012}, and notice that the lemma remains unchanged even if we add a constant $\beta / \lambda$. Thus, we expect, by our heuristic arguments, that the pair $(W, \beta)$ satisfies the differential equation
\begin{equation*} 
\mathcal{A}W(x) + \pi(x)= \beta.
\end{equation*}
These types of equations often arise in ergodic control problems and there is lots of literature on sufficient conditions for the existence of a solution to these equations, see \cite{Fleming1999, Robin1983, Borkar}. We remark here, that usually these conditions rely heavily on the solution of the corresponding discounted infinite horizon control problems, and thus apply the so-called vanishing discount method. However, in our case we can proceed by explicit calculations.

Next we shall determine the equation for the pair $(W, \beta)$ in the action region $[y^*, \infty]$. The Poisson process jumps in infinitesimal time with probability $\lambda dt$, and in that case the agent has to pay a cost $\gamma (x- y^*) + F(T, y^*)$. On the other hand, the Poisson process does not jump with probability $1- \lambda dt$, and in this case the agent has to pay  $\pi(x) dt + \mathbb{E}_x\left[ F(T-dt,X_{dt}) \right]$. Thus, the candidate function $F$ should satisfy the condition
\begin{align*}
F(T,x) = & \lambda dt(\gamma (x- y^*) + F(T, y^*)) \\ & + (1- \lambda dt)\left( \pi(x) dt + \mathbb{E}_x\left[ F(T-dt,X_{dt}) \right] \right).
\end{align*}
Now using again the heuristic \eqref{Heuristic} and that intuitively $dt^2=0$, we find that
$$
W(x) = \lambda dt (\gamma (x-y^*) + W(y^*)) + \pi(x)dt - \beta dt + (1- \lambda dt)\mathbb{E}_x \left[ W(X_{dt}) \right].
$$
By formally using Dynkin's formula to the last term and simplifying we get
$$
0 =  \lambda dt (\gamma (x-y^*) + W(y^*)) + \pi(x)dt - \beta dt + (\mathcal{A} - \lambda)W(x) dt. 
$$
We conclude that on the action region, the pair $(W, \beta)$ should satisfy the differential equation
\begin{equation*}
 (\mathcal{A} - \lambda)W(x)  = -(\pi(x) + \lambda  (\gamma (x-y^*) + W(y^*))-\beta).   
\end{equation*}
Now, we first observe that 
\begin{equation*}
\mathcal{A}W(x)+\pi(x)-\beta =
    \begin{cases}
    0, & x<y^*,\\
    \lambda(W(x)- \gamma x -(W(y^*)-\gamma y^*)), & x\geq y^*,
    \end{cases}
\end{equation*}
which implies that $W(x)$ satisfies the $C^1$-condition $W'(y^*)=\gamma$. We have thus arrived at the following free boundary problem: Find a function $W(x)$ and constants $y^*$ and $\beta$ such, that
\begin{align}
 W  \in  C^2, \nonumber \\
 W'(y^*)  =  \gamma,  \nonumber \\
 (\mathcal{A} - \lambda)W(x) + \pi(x) + \lambda  (\gamma (x-y^*) + W(y^*))  =  \beta, \, \, \, \, x \in [y^*, \infty), \label{equation above}\\
 \mathcal{A}W(x) + \pi(x)   = \beta,  \, \, \, \,  x \in (0,y^*). \label{equation below}
\end{align}
\begin{remark}
Another common approach to heuristically form the HJB-equation of the problem is to use the value function $J_r(x)$ of the corresponding discounted problem (see \cite{Lempa2014} p.4) and the vanishing discount limits $rJ_r(\bar{x}) \to \beta$ and $W_r(x) = J_r(x)-J_r(\bar{x}) \to W(x)$, where $\bar{x}$ is a fixed point in $\mathbb{R}_+$ (see \cite{Robin1983} p.284 and \cite{Fleming1999} p.427). This argument yields the HJB-equation
\begin{align}\label{HJB}
\mathcal{A}W(x) + \pi(x) - \lambda(W'(x)-\gamma)\mathbbm{1}_{\{ x \in S \}} = \beta,  
\end{align}
where $S = [y^*,\infty)$ is the control region. 
\end{remark}
To solve the free boundary problem, we consider first the equation (\ref{equation below}). In this case we write the differential operator $\mathcal{A}$ as (see \cite{Yor1997} p.285)
$$
\mathcal{A} = \frac{d}{dm(x)} \frac{d}{dS(x)},
$$
which allows us to find that
\begin{equation} \label{DY below}
\frac{d\, W'(x)}{d\,  S'(x)} = (\beta - \pi (x))m'(x).
\end{equation}
Therefore, integrating over the interval $(0, y^*)$ gives
\begin{equation*}
\frac{W'(y^*)}{S'(y^*)} = \beta \int_{0}^{y^*} m'(z)dz - \int^{y^*}_0 \pi(z)m'(z) dz. 
\end{equation*}
Hence, by the assumption \ref{assumption 3} and the $C^1$-condition $W'(y^*)=\gamma$, we get
\begin{equation*} 
\beta  = \bigg[\int_{0}^{y^*} m'(z)dz\bigg]^{-1} \left[ \int^{y^*}_{0} \pi(z)m'(z) dz + \frac{\gamma}{S'(y^*)} \right].
\end{equation*}
Finally, by using the formula (\ref{m integraali skaalana}), we arrive at
\begin{equation} \label{beta below}
\beta  = \bigg[\int_{0}^{y^*} m'(z)dz\bigg]^{-1} \left[ \int^{y^*}_{0} \pi_{\mu}(z)m'(z) dz \right].
\end{equation}

Next we consider the equation (\ref{equation above}). We immediately find that a particular solution is
\begin{equation*}
W(x) = (\R \pi_{\gamma})(x) - \frac{\beta}{\lambda} - \gamma y^* + W(y^*).
\end{equation*}
Hence, we conjecture analogously to \cite{Lempa2014} p.113, that the solution to (\ref{equation above}) is
\begin{equation} \label{W above}
W(x) = (\R \pi_{\gamma})(x) - \frac{\beta}{\lambda} - \gamma y^* + W(y^*) + C \varphi_{\lambda}(x).
\end{equation}
To find the constants $C$ and $\beta$, we first use the continuity of $W$ at the boundary $y^*$, which allows us to substitute $x=y^*$ to (\ref{W above}). This yields
\begin{equation} \label{ylapuolella yhtalo 1}
0 = (\R \pi_{\gamma})(y^*) - \frac{\beta}{\lambda} - \gamma y^* + C\varphi_{\lambda}(y^*).
\end{equation}
Then, by applying the condition $W'(y^*) = \gamma$ on (\ref{W above}), we find that
\begin{equation*}
C = \frac{\gamma - (\R \pi_{\gamma})'(y^*)}{\varphi_{\lambda}'(y^*)}.
\end{equation*}
Combining this with (\ref{ylapuolella yhtalo 1}) gives 
\begin{equation*} 
\beta = \lambda (\R \pi_{\gamma})(y^*)- \lambda \gamma y^* + \frac{\gamma - (\R \pi_{\gamma})'(y^*)}{\varphi_{\lambda}'(y^*)} \lambda \varphi_{\lambda}(y^*).
\end{equation*}
To re-write this expression, we first notice that a straightforward differentiation gives
\begin{align*}
\frac{d}{dx} \bigg( x \frac{\varphi_{\lambda}'(x)}{S'(x)} - \frac{\varphi_{\lambda}(x)}{S'(x)} \bigg) = -m'(x)\varphi_{\lambda}(x)(\mu(x)- \lambda x).
\end{align*}
Thus, by the fundamental theorem of calculus and the naturality of the upper boundary, we get
\begin{equation} \label{Eq: phin mu integraali}
y^* \frac{\varphi_{\lambda}'(y^*)}{S'(y^*)} - \frac{\varphi_{\lambda}(y^*)}{S'(y^*)} = \int_{y^*}^{\infty} m'(z)\varphi_{\lambda}(z)(\mu(z)- \lambda z) dz.
\end{equation}
Next, by using the formula (\ref{calculation formula for resolvent}), we find that
\begin{align*}
 & (\R \pi_{\gamma})(y^*)\varphi_{\lambda}'(y^*)   - (\R \pi_{\gamma})'(y^*) \varphi_{\lambda}(y^*) \\ & =
 - S'(y^*)\int_{y^*}^{\infty} \varphi_{\lambda}(z) \pi_{\gamma}(z)m'(z)dz. 
\end{align*}
Combining these observations with the formula (\ref{natural boundary phi integraali}), the constant $\beta$ reads as
\begin{align*}
\beta = & \bigg[ \int_{y^*}^{\infty} \varphi_{\lambda}(z)m'(z)dz \bigg]^{-1} \times \\
& \bigg[ \int_{y^*}^{\infty} \varphi_{\lambda}(z)  \pi_{\gamma}(z)m'(z)dz + \gamma \int_{y^*}^{\infty} m'(z)\varphi_{\lambda}(z)(\mu(z)- \lambda z) dz  \bigg]. \nonumber
\end{align*}
Finally, by recalling the definition of $\pi_{\mu}(x)$, we have
\begin{equation} \label{beta above}
\beta= \bigg[ \int_{y^*}^{\infty} \varphi_{\lambda}(z)m'(z)dz \bigg]^{-1} \bigg[ \int_{y^*}^{\infty} \varphi_{\lambda}(z) \pi_{\mu}(z)m'(z)dz \bigg].
\end{equation}

Now, by equating the representations (\ref{beta below}) and (\ref{beta above}) of $\beta$, we find the optimality condition
\begin{align*}
& \int^{y^*}_{0} \pi_{\mu}(z)m'(z) dz \int_{y^*}^{\infty} \varphi_{\lambda}(z)m'(z)dz \\ = & \int_{0}^{y^*} m'(z)dz  \int_{y^*}^{\infty} \varphi_{\lambda}(z) \pi_{\mu}(z)m'(z)dz,
\end{align*}
which can be re-expressed, using the functions $L(x)$ and $H(0,x)$, as
\begin{equation} \label{optimaalisuusyhtalo}
   m(0,y^*)L(y^*) = -\varphi_{\lambda}'(y^*)H(0,y^*).
\end{equation}
We proved in proposition \ref{uniqueness of the equation} that there exists a unique solution $\bar{x}$ to the condition (\ref{optimaalisuusyhtalo}), and thus, we will assume in the following that $y^*=\bar{x}$.

\begin{remark}
As often in ergodic optimal control problems, the potential value function $W(x)$ satisfies the second order differentiability across the boundary $\lim_{x \downarrow y^*} W''(x)= \lim_{x \uparrow y^*} W''(x)$. This can be verified as follows. When $x>y^*$, differentiation and the harmonicity properties $(R_{\lambda} (\mathcal{A-\lambda})\pi_{\gamma})(x)+\pi_{\gamma}(x) = 0$ and $(\mathcal{A}-\lambda)\varphi_{\lambda}(x) = 0$ yield
\begin{align*}
     \lim_{x \downarrow y^*}W''(x) & = (\R \pi_{\gamma})''(y^*) + C \varphi_{\lambda}''(y^*) \\
     & = \frac{2}{\sigma^2(y^*)} \bigg[ \lambda(R_{\lambda} \pi_{\gamma})(y^*) - \pi_{\gamma}(y^*) - \mu(y^*)(R_{\lambda} \pi_{\gamma})'(y^*) \\ & + \frac{\gamma-(R_{\lambda} \pi_{\gamma})'(y^*)}{\varphi_{\lambda}'(y^*)}(\lambda \varphi_{\lambda}(y^*)- \mu(y^*)\varphi_{\lambda}'(y^*))  \bigg], 
     \end{align*}
which after cancellation and formulas (\ref{Eq: phin mu integraali}) and (\ref{calculation formula for resolvent}) equals
\begin{align*}
     \frac{2}{\sigma^2(y^*)} \bigg[ -\frac{\lambda S'(y^*)}{\varphi_{\lambda}'(y^*)} \int_{y^*}^{\infty} \varphi_{\lambda}(z) m'(z) ( \pi_{\gamma}(z) - \gamma \lambda z + \mu(z) \gamma)dz - \pi_{\mu}(y^*)  \bigg]. 
\end{align*}
Therefore, by using the formulas (\ref{natural boundary phi integraali}) and (\ref{beta above}), we find that
\begin{align*}
     \lim_{x \downarrow y^* }W''(x) = \frac{2}{\sigma^2(y^*)} [ \beta - \pi_{\mu}(y^*)  ].
\end{align*}
On the other hand, when $x < y^*$, we notice that
\begin{equation*}
    \frac{d}{dx} \bigg[ \frac{W'(x)-\gamma}{S'(x)} \bigg] = (\mathcal{A}W'(x) - \gamma \pi(x))m'(x) = (\beta - \pi_{\mu}(x))m'(x).
\end{equation*}
Hence, by differentiating the left hand side and plugging in $x = y^*$, we find by the first order condition $W'(y^*) = \gamma$ that 
\begin{equation*}
    \lim_{x \uparrow y^* }W''(x) = \frac{2}{\sigma^2(y^*)} [ \beta - \pi_{\mu}(y^*) ].
\end{equation*}
\end{remark}

\subsection{Sufficient conditions}

We begin by an initial remark. When $x > y^*$, we get by differentiating (\ref{W above}) and using lemma \ref{lemma for variational} that
\begin{equation*}
    W'(x)-\gamma = \varphi'_{\lambda}(x)\bigg[ \frac{(R_{\lambda} \pi_{\gamma})'(x)-\gamma}{\varphi'_{\lambda}(x)} -\frac{(R_{\lambda} \pi_{\gamma})'(y^*)-\gamma}{\varphi'_{\lambda}(y^*)} \bigg] > 0.
\end{equation*}
In the opposite case, when $x < y^*$, we have
\begin{equation*}
    \frac{d}{dx} \bigg[ \frac{W'(x)-\gamma}{S'(x)} \bigg] = (\beta - \pi_{\mu}(x))m'(x).
\end{equation*}
Thus, by integrating over the interval $(0, x)$, and using the expression (\ref{beta below}) and lemma \ref{lemma for variational}, we find that
\begin{equation*}
    \frac{W'(x)-\gamma}{S'(x)} = m(0, x) \bigg[ \frac{\int_{0}^{y^*} \pi_{\mu}(t)m'(t)dt }{m(0, y^*)} - \frac{\int_{0}^{x} \pi_{\mu}(t)m'(t)dt }{m(0, x)} \bigg] < 0.
\end{equation*}
These observations imply, that under the standing assumptions, the function $W(x)-\gamma x$ has a global minimum at $y^*$ which shows that $W(x)$ satisfies the variational equality
\begin{equation} \label{variational equality}
    \mathcal{A}W(x)+\pi(x) + \lambda \big[ \inf_{y \leq x}\{(W(y)-\gamma y)-(W(x)-\gamma x)\} \big] = \beta.
\end{equation}

\begin{proposition}[Verification]
Under the assumptions 1, 2 and 3, the optimal policy is as follows. If the controlled process $X^{\zeta}$ is above the threshold $y^*$ at a jump time of $N$, i.e. $X^{\zeta}_{T_-} > y^*$, the decision maker should take the controlled process $X^{\zeta}$ to $y^*$. Further, the threshold $y^*$ is uniquely determined by (\ref{optimaalisuusyhtalo}), and the constant $\beta$ characterized by (\ref{beta below}) and (\ref{beta above}) gives the minimum average cost 
\begin{equation}
    \beta = \inf_{\zeta} \liminf_{T \to \infty} \frac{1}{T} \mathbb{E}_x \bigg[ \int_0^T (\pi(X_s)ds + \gamma d \zeta_s) \bigg].
\end{equation}
\end{proposition}
\begin{proof}
Define the function 
\begin{align*}
    \Phi(x):= & \inf_{y \leq x}\{(W(y)-\gamma y)-(W(x)-\gamma x)\} \\ = & \{W(y^*)-W(x)+\gamma(x-y^*) \} \mathbbm{1}_{[y^*,\infty)}(x),
\end{align*}
and a family of almost surely finite stopping times $\tau(\rho)_{\rho > 0}$ as $\tau(\rho):= \tau_0^\zeta \wedge \rho \wedge \tau^\zeta_{\rho}$ where $\tau_{\rho}^\zeta = \inf \{t \geq 0: X_t^\zeta \not \in (1/\rho, \rho) \}$. By applying the Doléans-Dade-Meyer change of variables formula to the process $W(X_t)$, we obtain
\begin{align*}
    W(X_{t \wedge \tau(\rho)}) - W(x) = & \int_0^{t \wedge \tau(\rho)} \mathcal{A}W(X_s^{\zeta})ds + \int_0^{t \wedge \tau(\rho)} \sigma(X_s^{\zeta})W'(X_s^{\zeta})dB_s \\ & + \sum_{0 \leq s \leq {t \wedge \tau(\rho)}} [W(X_s^{\zeta})-W(X_{s-}^{\zeta})]. 
\end{align*}
Because the control $\zeta$ jumps only if the Poisson process $N$ jumps, we have that $W(X_s^{\zeta}) - W(X_{s-}^{\zeta}) + \gamma(\Delta \zeta_s) \geq \Phi(X_{s-}^{\zeta})$. By combining these two observations with (\ref{variational equality}), we get
\begin{align} \label{W verification inequality}
    W(X_{t \wedge \tau(\rho)})  \geq & W(x) + \beta ({t \wedge \tau(\rho)}) - \int_0^{t \wedge \tau(\rho)} [\pi(X_s^{\zeta})ds + \gamma d \zeta_s] \\ & +  Z_{t \wedge \tau(\rho)} + M_{t \wedge \tau(\rho)}, \nonumber
\end{align}
where
\begin{equation*}
    M_t := \int_0^t \sigma(X_s^{\zeta})W'(X_s^{\zeta})dB_s, \quad \quad Z_t := \int_0^t \Phi(X_s^{\zeta})d\tilde{N}_s. 
\end{equation*}
Here $\tilde{N}_t = (N_t-\lambda t)_{t \geq 0}$ is the compensated Poisson process. It follows from the calculation above, that $Z_{t \wedge \tau(\rho)} + M_{t \wedge \tau(\rho)}$ is a submartingale and thus $\mathbb{E}_x [ Z_{t \wedge \tau(\rho)} + M_{t \wedge \tau(\rho)} ] \geq 0$. Taking expectation from both sides, dividing by ${t \wedge \tau(\rho)}$ and letting $t, \rho \to \infty$, we find that
\begin{equation*}
    \liminf_{T\to \infty} \frac{1}{T} \mathbb{E}_x \bigg[ W(X_T^{\zeta}) + \int_0^T (\pi(X_s^{\zeta}) ds + \gamma d\zeta_s) \bigg] \geq \beta.
\end{equation*} 
Thus, if $\liminf_{T\to \infty} \frac{1}{T} \mathbb{E}_x [ W(X_T^{\zeta})] = 0$, it follows that $J(x,\zeta) \geq \beta$; we postpone the proof of this limiting property to the following lemma.

Next, we prove that ${J(x,\zeta^*) \leq \beta}$. We proceed as above and note that (\ref{W verification inequality}) holds as equality when $\zeta = \zeta^*$. Hence, the local martingale term $M_T+Z_T$ is now uniformly bounded from below by $-W(x)-\beta T$, and therefore a supermartingale. Thus, we have that 
\begin{equation*}
    \mathbb{E}_x \bigg[\int_0^T (\pi(X_s^{\zeta}) ds + \gamma d\zeta^*_s) \bigg] \leq \beta T + W(x)- \mathbb{E}_x[W(X_T)] \leq \beta T + W(x).
\end{equation*}
Finally, dividing by $T$ and letting $T \to \infty$, we get
\begin{equation*}
    J(x,\zeta^*) \leq \beta,
\end{equation*}
which completes the proof.
\end{proof}

As usually in ergodic control problems, we noticed in the proof that the verification theorem holds under the assumption that $\liminf_{T\to \infty} \frac{1}{T} \mathbb{E}_x [ W(X_T^{\zeta})] = 0$. Thus, in the following lemma we give sufficient condition on $\pi(x)$ under which the limit equals zero.
\begin{lemma}  The limit
\begin{equation} \label{Limit 0 W}
    \liminf_{T\to \infty} \frac{1}{T} \mathbb{E}_x [ W(X_T^{\zeta})] = 0
\end{equation}
holds if
$$\pi(x) \geq C (x^{\alpha}-1),$$
where $\alpha$ and $C$ are positive constants.
\end{lemma}
\begin{proof}
Let $x > y^*$. Then $W(x)$ reads as
\begin{equation*} 
W(x) = (\R \pi_{\gamma})(x) - \frac{\beta}{\lambda} - \gamma y^* + W(y^*) + C \varphi_{\lambda}(x).
\end{equation*}
\begin{equation*} 
W(x) = (\R \pi_{\mu})(x) - \frac{\beta}{\lambda} + \gamma (x-y^*) + W(y^*) + C \varphi_{\lambda}(x).
\end{equation*}
Because $\varphi_{\lambda}(x)$ is bounded in this region, we only need to deal with the resolvent term. By Markov property and substitution $k=s+T$, we find that
\begin{align*}
    \mathbb{E}_x [(\R \pi_{\gamma})(X_T)] = & \int_0^{\infty}e^{-\lambda s} \mathbb{E}_x \big[ \mathbb{E}_{X_T} \big[  \pi_{\gamma}(X_s)  \big] \big] ds \\
    = & \int_0^{\infty}e^{-\lambda s} \mathbb{E}_x \big[ \mathbb{E}_{x} \big[  \pi_{\gamma}(X_{s+T}) \mid \mathcal{F}_T  \big] \big] ds \\
    = & \int_0^{\infty}e^{-\lambda s} \mathbb{E}_x \big[   \pi_{\gamma}(X_{s+T}) \big] ds \\
    = & e^{\lambda T} \int_T^{\infty}e^{-\lambda k} \mathbb{E}_x \big[   \pi_{\gamma}(X_{k}) \big] dk.
\end{align*}
Now by l'Hopitals rule and the assumption that id and $\pi$ are elements of $\mathcal{L}^{\lambda}_1$, we find that
\begin{equation} \label{limit of resolvent}
    \liminf_{T \to \infty}  \frac{e^{\lambda T}}{T} \int_T^{\infty}e^{-\lambda k} \mathbb{E}_x \big[   \pi_{\gamma}(X_{k}) \big] dk = \liminf_{T \to \infty} \frac{1}{T}  \mathbb{E}_{x} \big[   \pi_{\gamma}(X_{T}) \big].
\end{equation}
On the other hand, if 
$$
\liminf_{T \to \infty} \frac{1}{T} \mathbb{E}_x [X_T] > 0
$$ 
there exists $T_1$ such that 
$$
\mathbb{E}_x [X_s]  > \varepsilon \frac{s}{(\alpha+1)^{1/\alpha}}
$$
for all $s > T_1$. Together with the assumption
$\pi(x) \geq  C (x^{\alpha}-1)$ this leads to contradiction as
\begin{equation} \label{LemmaContradiction}
\begin{aligned}
   \infty > & \liminf_{T \to \infty} \frac{1}{T} \mathbb{E}_x \bigg[ \int_0^T (\pi(X_s)ds + \gamma d \zeta_s) \bigg] \\ \geq & \liminf_{T \to \infty} \frac{1}{T} \mathbb{E}_x \bigg[ \int_0^T \pi(X_s)ds \bigg] \\ \geq & -C + C \liminf_{T \to \infty} \frac{1}{T} \mathbb{E}_x \bigg[ \int_0^T X_s^{\alpha} ds \bigg] \\
   \geq & -C + C \liminf_{T \to \infty} \frac{1}{T} \mathbb{E}_x \bigg[ \frac{\varepsilon^{\alpha}}{\alpha+1}\int_0^T s^{\alpha} ds \bigg] \\
   = & -C + C \varepsilon^{\alpha} \liminf_{T \to \infty} T^{\alpha} = \infty.
\end{aligned}
\end{equation}
Similarly, we must have
$$
\liminf_{T \to \infty} \frac{1}{T} \mathbb{E}_x [\pi(X_T)] = 0,
$$ 
and thus conclude that the limit ($\ref{Limit 0 W}$) must vanish.

In the opposite case $x < y^*$, we find by integrating in (\ref{DY below}) that 
\begin{equation*}
    \frac{W'(y)}{S'(y)} = \frac{\gamma}{S'(y^*)} +  \int_y^{y^*} m'(z)(\pi(z)-\beta) dz.
\end{equation*}
Multiplying by $S'(x)$ and integrating over the interval $(x,y^*)$, we have
\begin{equation*}
    W(x) = W(y^*) - \frac{\gamma}{S'(y^*)} \int_x^{y^*}S'(z)dz - \int_x^{y^*} \int_y^{y^*} m'(z)(\pi(z)-\beta) dz S'(y)dy.
\end{equation*}
This has the upper bound (as the second term is negative and $\pi(x)$ is positive everywhere)
\begin{equation*}
    W(x) \leq W(y^*)  + \beta \int_x^{y^*} \int_y^{y^*} m'(z) dz S'(y)dy.
\end{equation*}
As the last integral is positive we can by assumption \ref{assumption 3} expand the region of the inner integral to get 
\begin{equation*}
    W(x) \leq W(y^*)  + \beta \int_x^{y^*} \int_0^{y^*} m'(z) dz S'(y)dy.
\end{equation*}
Thus,
\begin{equation*}
    W(x) \leq W(y^*)  + \beta m(0,y^*)(S(y^*)-S(x)).
\end{equation*}
Consequently the upper bound is of the form
\begin{equation*}
    W(x) \leq C_0 S(x) + C_1,
\end{equation*}
where $C_0$ and $C_1$ are constants. Since $S(X_t)$ is a non-negative local martingale (see \cite{Bass1998} p.88), and hence a supermartingale, we have
\begin{align*}
    \mathbb{E}_x[W(X_T)] \leq C_0 \mathbb{E}_x[S(X_T)] + C_1 & \leq C_0 \mathbb{E}_x[S(X_0)] + C_1 \\ & = C_0 S(x) + C_1.
\end{align*}
Hence, also in this case the limit ($\ref{Limit 0 W}$) must vanish.
\end{proof}

\begin{remark}
Another approach to see that the limit vanishes is to get a suitable upper bound for $W(x)$. Indeed, if $W(x) \leq A_0 + A_1\pi(x)$ for some constant $A_0$ and $A_1$, then the result also holds by lemma 3.1 in \cite{Wang2001}.
\end{remark}

\section{Ergodic Singular Control Problem: Connecting the Problems}

The singular control problem, where the agent is allowed to control the process $X_t$ without any constraints, is studied in the case of Brownian motion in \cite{Karatzas1983} and in the case of a more general one-dimensional diffusion in \cite{Alvarez2019, JackZervos2006}. In this corresponding singular problem the optimal policy is a local time reflecting barrier policy. The threshold $y^*_s$ characterizing the optimal policy is the unique solution to the optimality condition (see \cite{Alvarez2019} p.17)
\begin{equation} \label{singular problem}
H(0,y^*_s)=0.
\end{equation}
Heuristically one would expect that in the limit $\lambda \to \infty$, this optimal boundary $y^*_s$ coincides with the optimal boundary $y^*$. This is because in the limit, the decision maker has more frequent opportunities to exercise the control. This is shown in the next proposition after an auxiliary lemma. 
\begin{lemma} \label{limit phi}
Let $\varphi_{\lambda}(x)$ be the decreasing solution to the differential equation $(\mathcal{A}-\lambda)f=0$ and assume that $x < z$, then
\begin{equation}
\frac{\varphi_{\lambda}(z)}{\varphi_{\lambda}(x)} \xrightarrow{ \lambda \to \infty } 0.
\end{equation}
\end{lemma}
\begin{proof}
Taking the limit $\lambda \to \infty$ in (see \cite{BorodinSalminen} p.18)
\begin{equation*}
    \mathbb{E}_x [e^{-\lambda \tau_z}] = \frac{\varphi_{\lambda}(z)}{\varphi_{\lambda}(x)},
\end{equation*}
where $\tau_z = \inf \{ t \geq \mid X_t = z \}$ is the first hitting time to $z$, yields the result by monotone convergence.
\end{proof}

We now have the following result.
\begin{proposition}
Define a function $G:\mathbb{R}_+\to \mathbb{R}$ as 
$$
\hat{G}(x) = L(x)+ \frac{\varphi_{\lambda}'(x)}{S'(x)m(0,x)} H(0,x).
$$
Let $y^*$ and $y^*_s$ be the unique solutions to $\hat{G}(x)=0$ and $H(0,x) = 0$, respectively. Then $\hat{G}(y^*_s) \to 0$ as $\lambda \to \infty$.
\end{proposition}
\begin{proof}
Because $H(0,y^*_s)=0$, and the upper boundary $\infty$ is natural, we have
\begin{equation*}
    \frac{\hat{G}(y^*_s)}{\varphi_{\lambda}(y^*_s)} = \frac{L(y^*_s)}{\varphi_{\lambda}(y^*_s)} = \int_{y^*_s}^{\infty} \frac{\varphi_{\lambda}(z)}{\varphi_{\lambda}(y^*_s)} (\pi_{\mu}(z)-\pi_{\mu}(y^*_s)) m'(z)dz.
\end{equation*}
Thus taking the limit $\lambda \to \infty$ yields the result by lemma \ref{limit phi}.
\end{proof}

It is also reasonable to expect that when $\lambda$ increases, it is more likely that the decision maker postpones the exercise of the control as he has more information about the underlying process available. Therefore, we expect that the optimal threshold $y^*$ is increasing as a function of $\lambda$. The next proposition shows that this is indeed the case.

\begin{proposition}
Assume that $\mu(x)>0$. Then the unique root $y^*_{\lambda}$ of the function
\begin{equation*}
    G_{\lambda}(x) = \frac{L(x)}{\varphi_{\lambda}'(x)}+\frac{H(0,x)}{S'(x) m(0,x)}
\end{equation*}
is increasing in $\lambda$.
\end{proposition}
\begin{proof}
Let $\hat{\lambda}>\lambda$. From the proof of lemma \ref{limit phi}, we find that for every $x<z$
\begin{equation*}
    \frac{\varphi_{\lambda}(z)}{\varphi_{\lambda}(x)} \geq \frac{\varphi_{\hat{\lambda}}(z)}{\varphi_{\hat{\lambda}}(x)},
\end{equation*}
which is equivalent to
\begin{equation} \label{phin tavallisten jarjestys}
    \frac{\varphi_{\lambda}(z)}{\varphi_{\hat{\lambda}}(z)} \geq \frac{\varphi_{\lambda}(x)}{\varphi_{\hat{\lambda}}(x)}.
\end{equation}
Because $\hat{\lambda}>\lambda$, there exists  $r >0$ such that $\hat{\lambda} = \lambda + r$. Thus utilizing the fact that $(\mathcal{A}-\lambda)\varphi_{\lambda+r}= (\mathcal{A}-(\lambda+r))\varphi_{\lambda+r} + r\varphi_{\lambda+r} = r\varphi_{\lambda+r}$ with the Corollary 3.2 of \cite{Alvarez2004}, we have
\begin{equation*}
    \varphi_{\lambda}(x) \varphi_{\hat{\lambda}}'(x)  - \varphi_{\hat{\lambda}}(x) \varphi_{\lambda}'(x) = -r S'(x)  \int_{x}^{\infty} \varphi_{\lambda}(y)\varphi_{\lambda+r}(y)m'(y)dy \leq 0.
\end{equation*}
Reorganizing the above we get
\begin{equation} \label{phin tavallisista derivaattoihin}
    \frac{\varphi_{\lambda}(x)}{\varphi_{\hat{\lambda}}(x)} \geq \frac{\varphi_{\lambda}'(x) }{\varphi_{\hat{\lambda}}'(x)}. 
\end{equation}
Combining (\ref{phin tavallisten jarjestys}) and (\ref{phin tavallisista derivaattoihin}), we deduce that
\begin{equation*}
     \frac{\varphi_{\lambda}(z)}{\varphi_{\lambda}'(x)} \leq \frac{\varphi_{\hat{\lambda}}(z)}{\varphi_{\hat{\lambda}}'(x)}.
\end{equation*}
Hence, the function $G_{\lambda}(x)$ satisfies
\begin{align*}
     G_{\lambda}(x) = \int_{x}^{\infty} \frac{\varphi_{\lambda}(z)}{\varphi_{\lambda}'(x)} \pi_{\mu}(z) m'(z)dz + \frac{\int_0^x \pi_\mu(z)m'(z)dz}{S'(x)m(0,x)} \leq G_{\hat{\lambda}}(x).
\end{align*}
This implies that $y^*_{\lambda} \leq y^*_{\hat{\lambda}}$, as by (\ref{G positiivinen ennen y^*}), $G_{\lambda}(x)$ is positive in the interval $(0,y^*_{\lambda})$ and has unique root.
\end{proof}
\begin{remark}
The assumption that $\mu(x)>0$ is somewhat restricting and is there to guarantee that $\pi_{\mu}(x)>0$. It would be enough that $$\frac{L(y^*_{\lambda})}{\varphi_{\lambda}'(y^*_{\lambda})} \leq \frac{L(y^*_{\lambda})}{\varphi_{\hat{\lambda}}'(y^*_{\lambda})}.$$
It is often hard to show this exactly, however, in applications it can be verified numerically.
\end{remark}

\begin{remark}
Denote by $\beta_s$ the average cost per unit time of the singular problem. Then it also holds that $\beta \xrightarrow{\lambda \to \infty} \beta_s$ (see \cite{LempaSaarinen2019} p.12).
\end{remark}

\section{Illustrations}

\subsection{Verhulst-Pearl diffusion} We consider a standard Verhulst-Pearl diffusion
$$
dX_t = \mu X_t(1-\beta X_t)dt+\sigma X_t dW_t, \quad \, X_0 = x \in \mathbb{R}_r,
$$
where $\mu >0$, $\sigma >0$, $\beta > 0$. This diffusion is often used as a model for stochasticly fluctuating populations, see \cite{AlvarezHening2019, EvansHeningScheiber2015}. The scale density and speed measure are in this case
$$
S'(x) = x^{-\frac{2 \mu}{\sigma^2}} e^{\frac{2\mu \beta}{\sigma^2} x}, \quad m'(x) = \frac{2}{\sigma^2} x^{\frac{2 \mu}{\sigma^2}-2} e^{-\frac{2\mu \gamma}{\sigma^2} x}.
$$
We assume, that the cost $\pi(x) = x^2$ and $\gamma=1$. Hence, $\pi_{\mu}(x) = x^2 - x \mu (1-\beta x) $.
In this setting, we note that if $\mu > \sigma^2/2$ then
$$
m(0,x) = \frac{2}{\sigma^2}\bigg( \frac{\sigma^2}{2 \mu \beta}\bigg) ^{\frac{2 \mu}{\sigma^2}-1} \bigg( \bigg[ \Gamma \bigg( \frac{2\mu}{\sigma^2}-1 \bigg)- \Gamma \bigg( \frac{2\mu}{\sigma^2}-1,\frac{2 \mu \beta x}{\sigma^2} \bigg) \bigg] \bigg).
$$
The minimal excessive functions read as (see \cite{DayanikKaratzas2003} pp.201-203)
\begin{align*}
\varphi_{\lambda}(x) & = x^{\alpha_1} U \bigg(\alpha_1, 1+ \alpha_1 -\alpha_2, \frac{2 \mu \beta x}{\sigma^2}\bigg), \\ \quad \psi_{\lambda}(x) & =  x^{\alpha_1} M \bigg( \alpha_1, 1+\alpha_1-\alpha_2, \frac{2\mu \beta x}{\sigma^2} \bigg),
\end{align*}
where $U$ and $M$ are the Kummer's confluent hypergeometric functions of the second and first kind respectively, and
\begin{align*}
& \alpha_1 = \frac{1}{2}-\frac{\mu}{\sigma^2}+\sqrt{\bigg(\frac{1}{2}-\frac{\mu}{\sigma^2} \bigg)^2+\frac{2\lambda}{\sigma^2}}, \\
& \alpha_2 = \frac{1}{2}-\frac{\mu}{\sigma^2} - \sqrt{\bigg(\frac{1}{2}-\frac{\mu}{\sigma^2} \bigg)^2+\frac{2\lambda}{\sigma^2}}.
\end{align*}
We see that our assumptions are satisfied, and thus the result applies. Unfortunately, the equation (\ref{optimaalisuusyhtalo}) for the optimal threshold $y^*$, and the formula for the minimum average cost $\beta$ (\ref{beta above}), are complicated and therefore left unstated. However, we can illustrate the results numerically. In the table \ref{table1} the optimal threshold $y^*$ is calculated with the values $\mu = 1$, $\sigma=1$ and $\beta = 0.01$ for a few different values of $\lambda$. We see from the table that as $\lambda$ increases the threshold $y^*$ gets closer to the corresponding threshold $y^*_s \approx 0.743$ of the singular control problem (\ref{singular problem}). 

\begin{table} 
    \centering
    \begin{tabular}{c|c}
         $\lambda$ & $y^*$  \\ \hline 
         5 & 0.317 \\
         10 & 0.496 \\
         50 & 0.656 \\
         100 & 0.684 \\
         1000 & 0.726
    \end{tabular}
    \caption{The values for the optimal threshold $y^*$ for some values of the intensity $\lambda$.}
    \label{table1}
\end{table}

\subsection{Standard Ornstein-Uhlenbeck process}

As remarked in the introduction, the results hold also for $\mathbb{R}$ with straightforward changes. Indeed, we only have to adjust the assumptions slightly, by changing the lower boundary from $0$ to $-\infty$ in assumptions \ref{assumption 2}, \ref{assumption 3} and change all the formulas accordingly. With this change we can study a larger class of processes.

Consider dynamics that are characterized by a stochastic differential equation 
$$
dX_t = - \beta X_t dt + dW_t, \quad \, X_0=x,
$$
where $\beta > 0$. This diffusion is often used to model continuous time systems that have mean reverting behaviour. To illustrate the results we choose the running cost
$
\pi(x) = |x|,
$
and consequently $\pi_{\mu}(x) = |x|-\gamma \beta x$. The scale density and the density of speed measure are in this case
$$
S'(x)= \exp( \beta  x^2) , \quad
m'(x)=2 \exp(- \beta  x^2),
$$
and the minimal excessive functions read as (see \cite{BorodinSalminen} pp.141)
$$
\varphi_{\lambda}(x) = e^{\frac{\beta x^2}{2}} D_{-\lambda / \beta}(x \sqrt{2 \gamma}), \quad \, \psi_{\lambda}(x) = e^{\frac{\beta x^2}{2}} D_{-\lambda / \beta}(-x \sqrt{2 \gamma}),
$$ 
where $D_{\nu}(x)$ is a parabolic cylinder function. The equation (\ref{optimaalisuusyhtalo}) for the optimal threshold takes again rather complicated form and thus the results are only illustrated numerically in table \ref{table2}. In the singular control case the equation (\ref{singular problem}) gives $y^*_s \approx 0.535$. Thus, as expected, the threshold value $y^*$ gets closer to $y^*_s$ when $\lambda$ increases.

\begin{table} 
    \centering
    \begin{tabular}{c|c}
         $\lambda$ & $y^*$  \\ \hline 
         1 & 0.182 \\
         5 & 0.301 \\
         10 & 0.353 \\
         100 & 0.469\\
         300 & 0.496 \\
    \end{tabular}
    \caption{The value of the optimal threshold $y^*$ for controlled Ornstein-Uhlenbeck process for $\beta=0.1$ and few choices of $\lambda$.}
    \label{table2}
\end{table}

\section{Conclusions}

We considered ergodic singular control problems with constraint of a regular one-dimensional diffusion. Relying on basic results from the classical theory of linear diffusions, we characterized the state after which the decision maker should apply an impulse control to the process. Our results are in agreement with the findings of \cite{Alvarez2019}, where the corresponding unconstrained singular control problem is studied. Indeed, no second order or symmetry properties of the cost are needed. In addition, we proved that as the decision maker gets more frequent chances to exercise the control, the value of the problem converges to that of the singular problem. 

There are few directions that the constrained problems could be studied further. To the best of our knowledge, the finite horizon problem with constraint remains open, even for the case of Brownian motion. Thus, it would be interesting if a similar analysis as in \cite{Karatzas1983} could be extended to also cover this case. In this case, we would expect a similar connections between the finite horizon time and the present problem as for those without any constraints (see \cite{Karatzas1983} p.241).

Moreover, the related two-sided problem, where the decision maker could control both downwards and upwards, but only at jump times of Poisson process, could be studied. Unfortunately, these extension are outside the scope of the present study, and therefore, left for future research.

\subsection*{Acknowledgements} 
\thispagestyle{empty}

Two anonymous referees are acknowledged for helpful comments. We would like to gratefully acknowledge the emmy.network foundation under the aegis of the Fondation de Luxembourg, for its continued support.

\end{document}